\documentclass[a4paper]{article}
\usepackage[utf8]{inputenc}

\usepackage[margin=1in]{geometry}
 \makeatletter
 \let\@fnsymbol\@arabic
 \makeatother

\usepackage{amsmath, amssymb}
\usepackage[hidelinks]{hyperref}

\usepackage{xcolor}

\usepackage{amsthm}
\usepackage[noabbrev]{cleveref}
\newtheorem{thm}{Theorem}[section]
\newtheorem{lm}[thm]{Lemma}

\newtheorem{prop}[thm]{Proposition}

\theoremstyle{definition}

\newtheorem{rmk}[thm]{Remark}

\newcommand{\eps}{\varepsilon}
\renewcommand{\phi}{\varphi}

\newcommand{\RR}{\mathbb R}
\newcommand{\FF}{\mathbb F}
	\newcommand{\fq}{\FF_q}

\renewcommand{\leq}{\leqslant}
\renewcommand{\geq}{\geqslant}

\newcommand{\vspan}[1]{\left \langle #1 \right \rangle}

\newcommand{\sett}[2]{ \left\{ #1 \, \, || \, \, #2 \right \} }

\newcommand{\one}{\mathbf 1}

\newcommand{\ma}{\mathcal A}
\newcommand{\mb}{\mathcal B}
\newcommand{\mc}{\mathcal C}

\newcommand{\mg}{\mathcal G}

\newcommand{\ml}{\mathcal L}

\newcommand{\mo}{\mathcal O}
\renewcommand{\mp}{\mathcal P}
\newcommand{\mq}{\mathcal Q}

\newcommand{\ms}{\mathcal S}
\newcommand{\mt}{\mathcal T}

\newcommand{\mx}{\mathcal X}
\newcommand{\my}{\mathcal Y}
\newcommand{\mz}{\mathcal Z}

 \DeclareMathOperator{\PG}{PG}
    \newcommand{\pg}{\PG}

 \DeclareMathOperator{\PGL}{PGL}
    \newcommand{\pgl}{\PGL}

 \DeclareMathOperator{\PGamL}{P\Gamma L}

 \DeclareMathOperator{\PGO}{PGO}

 \DeclareMathOperator{\NO}{NO}
  \newcommand{\nono}{\overline{\NO^+}(8,2)}
 \DeclareMathOperator{\Sym}{Sym}
 \newcommand{\BP}{\boldsymbol P}
 \DeclareMathOperator{\SRG}{SRG}
 \DeclareMathOperator{\Aut}{Aut}
 \DeclareMathOperator{\Johnson}{J}
 \DeclareMathOperator{\GO}{GO}
 \DeclareMathOperator{\OS}{OS}
 \DeclareMathOperator{\Alt}{Alt}
 \DeclareMathOperator{\ASL}{ASL}
 \DeclareMathOperator{\Inv}{Inv}

\title{A sporadic strongly regular graph with parameters $(120,56,28,24)$ from a primitive action of \\ the symmetric group on $7$ elements}

\author{
Sam Adriaensen
   \thanks{Department of Mathematics and Data Science, Vrije Universiteit Brussel, Pleinlaan 2, 1050 Elsene, Belgium.
   \href{mailto:Sam.Adriaensen@vub.be}{Sam.Adriaensen@vub.be}, \href{mailto:Jan.De.Beule@vub.be}{Jan.De.Beule@vub.be}}
\and Robert F.\ Bailey
   \thanks{School of Science and the Environment, Grenfell Campus, Memorial University, Corner Brook, NL, A2H 6P9, Canada. \href{mailto:robert.bailey@mun.ca}{Robert.Bailey@mun.ca}}
\and Jan De Beule
   \footnotemark[1]
\and Morgan Rodgers
   \thanks{Department of Mathematics, RPTU Kaiserslautern-Landau, Gottlieb-Daimler-Straße 48, 67663 Kaiserslautern, Germany.
   \href{mailto:Morgan.Rodgers@rptu.de}{Morgan.Rodgers@rptu.de}
   }
}
\date{}

\begin{document}

\maketitle

\begin{abstract}
 There are up to isomorphism exactly three strongly regular graphs with parameters $(120,56,28,24)$ whose automorphism group acts primitively on the vertices.
 Two of these graphs belong to classical families:
 one is the non-orthogonality graph on anisotropic points of the hyperbolic quadric $\mathcal Q^+(7,2)$, and the other one belongs to the Johnson scheme.
 The third one is not well understood.
 In this paper, we give a description of this graph in terms of ovoids and spreads of $\mathcal Q^+(7,2)$, or equivalently in terms of overlarge sets of Steiner systems with parameters $(3,4,8)$.
\end{abstract}

\section{Introduction}


A graph $\Gamma$ is called \emph{strongly regular} with parameters $(n,k,\lambda,\mu)$ if it has $n$ vertices, is $k$-regular, and given two distinct vertices, the number of common neighbours is given by $\lambda$ if they are adjacent, and $\mu$ if they are not adjacent.
We say that $\Gamma$ is an $\SRG(n,k,\lambda,\mu)$.
A strongly regular graph is said to be \emph{primitive} if both it and its complement (which is necessarily also strongly regular) are connected, and \emph{imprimitive} otherwise. It is well-known that the only imprimitive SRGs are complete multipartite graphs and their complements.  We note that a primitive SRG has diameter 2; SRGs are essentially the diameter-2 case of {\em distance-regular graphs} (see \cite{BCN}).  For further background on strongly regular graphs, we refer the reader to the recent monograph of Brouwer and Van Maldeghem \cite{Brouwer:VanMaldeghem}.

A permutation group $G$ acting on a set $V$ is \emph{transitive} if it has only one orbit on $V$.  A transitive action is said to be \emph{primitive} if the only equivalence relations on $V$ preserved by $G$ are the trivial or universal relations (i.e.\ where the equivalence classes are singletons or the whole of $V$, respectively), and is said to be \emph{imprimitive} otherwise.  For further background information on permutation groups, see \cite{Cameron}.
We note that if $G\leq\Aut(\Gamma)$ for some strongly regular graph $\Gamma$ and $G$ acts primitively on the vertex set $V(\Gamma)$, then $\Gamma$ is primitive (in the strongly regular graph sense), but the converse is not true in general: it is possible for a SRG $\Gamma$ to be primitive but for $\Aut(\Gamma)$ to be imprimitive (in the permutation group sense).

In this paper, we consider strongly regular graphs $\Gamma$ where a subgroup $G \leq \Aut(\Gamma)$ of automorphisms of $\Gamma$ acts primitively on $V(\Gamma)$.  Then $G$ naturally acts on the set $V(\Gamma)^2$ of ordered pairs of vertices, and the orbits of this action are the \emph{orbitals} of the action of $G$ of $V(\Gamma)$.
The number of orbitals is called the \emph{rank} of the action of $G$ on $V(\Gamma)$.
Equivalently, given any vertex $x \in V(\Gamma)$, the rank equals the number of orbits of the action of the stabiliser $G_x$ of $G$ acting on $V(\Gamma)$.
If $\Gamma$ is neither complete nor empty, the rank of $G$ has to be at least 3.
Graphs with a rank 3 automorphism group are necessarily strongly regular, and these have all been classified: see \cite{Brouwer:VanMaldeghem} for details.

More specifically, we are interested in $\SRG(120,56,28,24)$'s.  There are many known constructions of such graphs, including one-off examples, members of infinite families, or prolific constructions yielding multiple non-isomorphic graphs with the same parameters.  For instance, there is an imprimitive vertex-transitive example obtained by Brouwer, Ivanov, and Klin \cite{Klin}, the non-collinearity graphs of partial geometries with parameters $(7,8,4)$ (such as in \cite{Cohen}; see e.g.\ \cite[\S 8.6]{Brouwer:VanMaldeghem} for background), and the prolific constructions of Wallis \cite{Wallis} and of Goethals and Seidel \cite{Goethals:Seidel}.
In addition, one can obtain extra graphs through switching, see e.g.\ Schmidt \cite{Schmidt}, who used switching to obtain several pairs of $\SRG(120,56,28,24)$'s (or rather their complements) that are only isomorphic in the quantum sense, but not in the classical sense.

However, there are only three $\SRG(120,56,28,24)$'s with a primitive automorphism group.  This was verified by the second author and his students using GAP \cite{GAP4} and GRAPE \cite{GRAPE} using the following approach, as part of a larger analysis.  GAP contains libraries of all primitive groups on up to $4095$ points, obtained by the efforts of several authors (see  \cite{Coutts,Dixon:Mortimer,RoneyDougal,RoneyDougal:Unger}).  Also, the GRAPE package contains a function (which utilises methods developed by Praeger and Soicher \cite{Praeger:Soicher}), \mbox{\texttt{VertexTransitiveDRGs}}, which for a given transitive permutation group $G$ determines the distance-regular (and thus strongly regular) graphs $\Gamma$ with $G\leq\Aut(\Gamma)$, in terms of which combinations of orbitals of $G$ yield the edge set of each possible $\Gamma$.  By applying this function to each primitive group from the library on 120 points, and eliminating duplicates, it was determined that only three such examples with parameters $(120,56,28,24)$ exist.  (This can also be seen in \cite[Table 11.9]{Brouwer:VanMaldeghem}: while this table considers $\Aut(\Gamma)$ with rank at most $10$, from the primitive group library we know that the largest rank of a primitive group on $120$ points is $8$.)
Two of them are classical and well understood:
\begin{enumerate}
 \item The graph $\nono$, defined on the set of anisotropic points of $\pg(7,2)$ with respect to a hyperbolic quadratic form $f$.
 Two vertices are adjacent if they are not orthogonal for the alternating bilinear form associated to $f$.
 The full automorphism group is $\PGO^+(8,2)$, with rank 3.
 \item The distance 1-or-3 graph of the Johnson graph $\Johnson(10,3)$, where the vertices are the 3-element subsets of $\{1,\dots,10\}$, and where two vertices are adjacent if they share either 2 or 0 elements.
 The full automorphism group is the symmetric group $\Sym_{10}$ in its natural action on subsets, and the rank of this action is 4.
\end{enumerate}

The third graph, which we denote by $\Gamma^*$, is listed in \cite[Table 11.9]{Brouwer:VanMaldeghem}, but no more information is presented.
Its automorphism group $G^*$ is isomorphic to the symmetric group $\Sym_7$, and the stabiliser of a vertex has orbits of sizes $1, 7, 14, 14, 21, 21, 42$.
The purpose of this paper is to give a description of this graph.

The action of $G^*$ on $V(\Gamma^*)$ is generously transitive and hence yields a $6$-class association scheme.
$\Gamma^*$ can then be obtained as a fusion of certain relations in this scheme.
Interestingly, the graph $\nono$ can also be obtained as a fusion in the same scheme.

We present two descriptions of the association scheme.
One which emphasises how $\Sym_7$ arises as automorphism group, and one which emphasises how $\nono$ can be obtained as a fusion of the scheme:
\begin{enumerate}
 \item In the first description, the vertices of the scheme are certain overlarge sets of Steiner systems $\OS(4,8)$.
 This is a partition $\ms$ of $\binom A 4$, where $A$ is a 9-element set, into 9 classes $\sett{\mb_a}{a \in A}$ such that each $\mb_a$ constitutes a Steiner $(3,4,8)$-system on $A \setminus \{a\}$.
 We choose a two-element set $B \subset A$, which allows each $\ms$ to be split into $\ms_B = \sett{\mb_a}{a \in B}$ and $\ms_{\overline B} = \sett{\mb_a}{a \in A \setminus B}$.
 The relation containing a pair $(\ms,\mt)$ is then determined by the intersections between $(\ms_B, \ms_{\overline B})$ and $(\mt_B, \mt_{\overline B})$.
 The action of $G^*$ on the elements of $A \setminus B$ is then given by the natural action of $\Sym(A \setminus B) \cong \Sym_7$.
 \item By the work of Cameron and Praeger \cite{Cameron:Praeger}, there is a one-to-one correspondence between $\OS(4,8)$'s and spreads of $\mq^+(7,2)$.
 Using the principle of triality, we can translate the association scheme as being defined on a set of ovoids of $\mq^+(7,2)$.
 We find a one-to-one correspondence between these ovoids and the anisotropic points of $\pg(7,2)$.
 Hence, we find a description of the association scheme as being defined on the set of vertices of $\nono$, which explains why $\nono$ arises as a fusion in this scheme.
\end{enumerate}

We will also check that $\Gamma^*$ cannot be obtained from the constructions of Wallis \cite{Wallis} or Goethals--Seidel \cite{Goethals:Seidel}, or from a partial geometry.

A large part of our analysis will be done by the aid of the software GAP \cite{GAP4}, and the GAP packages GRAPE \cite{GRAPE} and FinInG \cite{fining}.
Results obtained through computation will be labelled by ``GAP''. All GAP code has been collected in one file and is available online via a github archive: see \cite{ABDBRgap}.

\section{Preliminaries}

We denote the projective space arising from an $(n+1)$-dimensional vector space $V$ over the finite field $\fq$ of $q$ elements by $\pg(n,q)$.
Given a vector $(x_0, \dots, x_n) \in \fq^{n+1}$, we denote by $(x_0: \cdots: x_n)$ the projective point $\vspan{(x_0,\dots,x_n)}$ in $\pg(n,q)$.

We will give an overview of results on quadrics in finite projective spaces and association schemes.
The reader can consult \cite{Brouwer:VanMaldeghem} as a reference.

\subsection{Quadrics}

A \emph{quadric} in $\pg(n,q)$ is the zero locus $\mq$ of a homogeneous quadratic form $f(X)$, where $X$ denotes a vector of variables.
Related to $f$ is the symmetric bilinear form $B(X,Y) = f(X+Y) - f(X) - f(Y)$.
A point is called \emph{singular} if its coordinate vectors $x$ satisfy the property that $Y \mapsto B(x,Y)$ is the zero function.
The quadric $\mq$ is called \emph{non-singular} if none of the points on $\mq$ are singular.
The set of singular points of $\mq$ forms a subspace $\rho$ of $\pg(n,q)$, and $\mq$ is a cone with vertex $\rho$, and as base a non-singular quadric.

If there are no singular points, then we can use the bilinear form $B$ to take orthogonal complements.
We map a subspace $U$ of $V$ to
\[
 U^\perp = \sett{v \in V}{(\forall u \in U)(B(u,v) = 0)}.
\]
This induces an involution on the subspaces of $\pg(n,q)$ that reverses inclusion, and maps $k$-spaces to $(n-k-1)$-spaces.
We call $\perp$ the \emph{polarity} associated to $\mq$.
Given two subspaces $\pi$ and $\rho$, we write $\pi \perp \rho$ if $\pi \subseteq \rho^\perp$ or equivalently $\rho \subseteq \pi^\perp$.
A subspace $\pi$ is called \emph{totally isotropic} if it is completely contained in the quadric, in which case $\pi \perp \pi$.
A point is called \emph{isotropic} or \emph{anisotropic}, according to whether or not it is on the quadric.
If $P, Q \in \mq$, then the line $PQ$ is totally isotropic if and only if $P \perp Q$.
In that case, we say that $P$ and $Q$ are \emph{collinear} points of $\mq$.
If $q$ is even (and $n$ is odd), then $B$ is alternating, and every point is orthogonal to itself according to $\perp$, regardless of whether it is on the quadric.
In that case, for any two points $P$ and $Q$, we have $P \perp Q$ if and only if the line $PQ$ is either totally isotropic, or intersects the quadric in a unique point.

Suppose that $d$ is the maximum projective dimension of the totally isotropic subspaces of $\mq$.
Then the totally isotropic $d$-spaces are called \emph{generators}, and the \emph{rank} of $\mq$ is $d+1$.
The quadrics over finite fields have been classified.
Up to isomorphism, in $\pg(n,q)$ there are $2$ non-singular quadrics if $n$ is odd, namely the elliptic quadric $\mq^-(n,q)$ of rank $\frac{n-1}2$ and the hyperbolic quadric of rank $\frac{n+1}2$, and a unique non-singular quadric if $n$ is even, namely the parabolic quadric $\mq(n,q)$ of rank $\frac n2$.
If $\mq$ is a non-singular quadric, then none of the anisotropic points are singular, unless $\mq$ is parabolic and $q$ is even, in which case there is a unique singular point called the \emph{nucleus} of $\mq(n,q)$.

Given a quadratic form $f$ on the vector space $V \cong \fq^{n+1}$ defining the quadric $\mq^\eps(n,q)$, with $\eps$ either $-$, $+$, or the empty character, we denote by $\GO^\eps(n+1,q)$ the group of invertible linear operators $\phi$ on $V$ such that $f \circ \phi = \lambda f$ for some non-zero scalar $\lambda$.
We denote by $\PGO^\eps(n+1,q)$ the group of collineations in $\pgl(n,q)$ arising from the action of $\GO^\eps(n+1,q)$ on $V$.

A set of pairwise non-collinear points is called a \emph{partial ovoid} of $\mq$.
A partial ovoid contains at most one point of each generator.
If it contains a point of each generator, it is called an \emph{ovoid}.
A set of pairwise disjoint generators is called a \emph{partial spread} of $\mq$.
Every point of $\mq$ is on at most one generator of a partial spread.
If every isotropic point is on a generator of the partial spread, i.e.\ if the partial spread partitions the isotropic points, we call it a \emph{spread}.

\subsection{Systems of generators and the triality}

The hyperbolic quadric $\mq^+(n,q)$ has several remarkable properties.
It generators can be split into two systems: generators $\pi$ and $\rho$ belong to the same system if and only if $\dim \pi \cap \rho \equiv \frac{n-1}2 \pmod 2$.
We note that $\PGO^+(n+1,q)$ acts transitively on the generators of $\mq^+(n,q)$, and the systems form blocks of imprimitivity.

In the case $n = 7$, there is a special map associated to $\mq^+(n,q)$ called the \emph{triality} map.
Let $\mp$ and $\ml$ be the sets of points and lines respectively on $\mq^+(7,q)$, and $\mg_1, \mg_2$ the systems of generators.
Then the triality map $\tau$ is a permutation of order 3 on $\mp \cup \ml \cup \mg_1 \cup \mg_2$ that maps $\mp$ to $\mg_1$, $\mg_1$ to $\mg_2$, $\mg_2$ to $\mp$, $\ml$ to itself, and respects incidence.
Note that $\tau$ permutes ovoids, spreads of generators in $\mg_1$, and spreads of generators in $\mg_2$.
Ovoids and spreads of $\mq^+(7,q)$ have size $q^3+1$.

\subsection{Association schemes}

Consider a finite set $\mx$, and a relation $R \subseteq \mx^2$ on $\mx$.
The \emph{adjacency matrix} of $R$ is defined as the matrix $A \in \RR^{\mx \times \mx}$ with $A(x,y) = 1$ if $(x,y) \in R$, and $A(x,y) = 0$ otherwise.
Now consider a set $\{R_0, \dots, R_d\}$ of relations on $\mx$, and let $A_i$ denote the adjacency matrix of $R_i$.
Then we call $(R_0, \dots, R_d)$, or equivalently $\ma = (A_0, \dots, A_d)$, a ($d$-class) \emph{symmetric association scheme} if
\begin{enumerate}
 \item $A_0 = I$,
 \item $A_0 + \ldots + A_d$ is the all-one matrix, which we denote by $J$,
 \item all matrices in $\ma$ are symmetric,
 \item the vector subspace $\RR[\ma]$ of $\RR^{\mx \times \mx}$ generated by $\ma$ is a commutative algebra for the ordinary matrix product.
\end{enumerate}

If $\ma = (A_0, \dots, A_d)$ is a symmetric association scheme, then $A_0, \dots, A_d$ can be diagonalised simultaneously.
$\RR^\mx$ admits an orthogonal decomposition $W_0 \oplus \dots \oplus W_d$ such that every eigenspace of every $A_i$ is the sum of some of the $W_j$ spaces.
Moreover, $W_0$ is spanned by the all-one vector.
The decomposition is unique up to permutation of $W_1, \dots, W_d$.
The $(d+1) \times (d+1)$-matrix $\BP$ where $\BP(j,i)$ is the eigenvalue of $A_i$ on the space $W_j$ is called the \emph{matrix of eigenvalues} of $\ma$.

The action of a group $G$ on $\mx$ is called \emph{generously transitive} if for each $x$ and $y$ in $X$, there exists some $g \in G$ with $x^g = y$ and $y^g = x$.
Equivalently, $G$ acts transitively, and its orbitals are symmetric.
In that case, the orbitals of $G$ yield a asymmetric $d$-class association scheme, where $d+1$ is the rank of $G$ acting on $\mx$.

\bigskip

If $(R_0, \dots, R_d)$ is an association scheme with eigenvalue matrix $\BP$, and we take a union of relations $R = R_{i_1} \cup \dots \cup R_{i_t}$, then $A = A_{i_1} + \ldots + A_{i_t}$ is the adjacency matrix of the graph $\Gamma$ on $\mx$ with adjacency relation $R$.
Its eigenvalues are given by $\BP (e_{i_1} + \ldots + e_{i_t})$, where $e_i$ denotes the $i$\textsuperscript{th} standard basis vector of $\RR^{d+1}$.
Since every relation of an association scheme must be regular, $\Gamma$ is a regular graph.
It is well-known that a regular graph is strongly regular if and only if it has at most 3 distinct eigenvalues.
Thus, $\Gamma$ is strongly regular if and only if $\BP (e_{i_1} + \ldots + e_{i_t})$ has at most 3 distinct entries.
If these entries are $k > \theta_1 \geq \theta_2$, then the strongly regular graph has parameters $\SRG(|\mx|,k, \theta_1 + \theta_2 + \theta_1 \theta_2 + k,\theta_1 \theta_2 + k)$.

\section{Overlarge sets of Steiner systems and spreads \texorpdfstring{of $\mq^+(7,2)$}{}}

\subsection{Overlarge sets of Steiner systems}

Suppose that $A$ is a set of $v$ points.
Denote by $\binom A k$ the set of subsets of $A$ of size $k$.
Consider a set $\mb \subseteq \binom Ak$, and call the elements of $\mb$ \emph{blocks}.
Then $\mb$ is called a \emph{$t$-$(v,k,\lambda)$ design} if any $T \in \binom At$ is contained in exactly $\lambda$ blocks.
If $\lambda=1$, the design is also called a \emph{Steiner system} with parameters $(t,k,v)$.

Now suppose that $A$ has size $v+1$, and suppose that we can partition the sets of $\binom A k$ into $v+1$ classes $\sett{\mb_a}{a \in A}$, such that each $\mb_a$ is a $(k-1)$-$(v,k,1)$ design on $A \setminus \{a\}$.
We call $\sett{\mb_a}{a \in A}$ an \emph{overlarge set of Steiner systems}, and denote it as an $\OS(k,v)$.

There are up to isomorphism two non-isomorphic $\OS(4,8)$.
This was first proven by Breach and Street \cite{Breach:Street}, and reproven by
Cameron and Praeger \cite{Cameron:Praeger} in a particularly elegant way.
Let us recall their results.
\begin{enumerate}
 \item Take $A$ to be the set of points of $\PG(1,8)$.
 Define
 \[
  \mb_{(1:0)} = \sett{ \{(x_1:1), \dots, (x_4:1) \} }{\{x_1,\dots,x_4\} \in \binom {\FF_8} 4, \, x_1 + \ldots + x_4 = 0 \}}.
 \]
 For each point $P$ of $\pg(1,8)$, let $\mb_P$ be the image of $\mb_{(1:0)}$ of a collineation $\phi \in \PGamL(2,8)$ that maps $(1:0)$ to $P$.
 The set $\mb_P$ is independent of the choice of $\phi$.
 Then $\sett{\mb_P}{P \in A}$ is an $\OS(4,8)$ on $A$.
 Its automorphism group is $\PGamL(2,8)$.

 \item Take $A = \FF_3^2$. Define
 \[
  \mb_0 = \sett{\{\pm v, \pm w\}}{v,w \in \FF_3^2 \text{ linearly independent}} \cup \sett{\{v,w,w+v,w-v\}}{v_1 w_2 - v_2 w_1 = 1}.
 \]
 Let $\mb_v$ be image of $\mb_0$ under the translation $w \mapsto w + v$ for each $v \in \FF_3^2$.
 Then $\sett{\mb_v}{v \in \FF_3^2}$ is an $\OS(4,8)$.
 Its automorphism group is $\ASL(2,3)$.
\end{enumerate}
There are $| \Sym_9 : \PGamL(2,8) | = 240$ $\OS(4,8)$'s of type (1), and $|\Sym_9 : \ASL(2,3) | = 1680$ of type (2).
Since both the $\PGamL(2,8)$ and $\ASL(2,3)$ seen as subgroups of $\Sym_9$ are subgroups of the alternating group $\Alt_9$, there are two $\Alt_9$ orbits of size $120$ on the $\OS(4,8)$'s of type (1), and two $\Alt_9$ orbits of size $840$ on the $\OS(4,8)$'s of type (2).
Moreover, both $\PGamL(2,8)$ and $\ASL(2,3)$ act 2-transitively.
Hence, if $G$ is the setwise stabiliser in $\Alt_9$ of a 2-element subset, then $G$ has the same orbits as $\Alt_9$ on the $\OS(4,8)$'s.

\subsection{Spreads \texorpdfstring{of $\mq^+(7,2)$}{}}

Following Cameron and Praeger \cite{Cameron:Praeger}, we can represent $\mq^+(7,2)$ using the following vector space $V$ and quadratic form $f$:
\begin{align}
 \label{Eq:CP}
 V = \sett{(x_1, \ldots, x_9) \in \FF_2^9}{x_1 + \ldots + x_9 = 0} \cong \FF_2^8,
 && f(X) = \sum_{1 \leq i < j \leq 9} X_i X_j.
\end{align}
The \emph{weight} of a vector is defined as the number of coordinate positions with a non-zero entry.
If $x \in V$ has weight $n$, then over $\FF_2$,
\[
 f(x) = \binom n2 = (n-1) \frac n2 = \frac n2.
\]
Thus, $\vspan x$ is an isotropic point if and only if $x$ has weight 4 or 8.
Note that the points of weight 8 form an ovoid.

Let $A = \{1, \dots, 9\}$.
We can identify each vector $x \in \FF_2^9$ with the set $\sett{a \in A}{x_a = 1}$.
Cameron and Praeger \cite{Cameron:Praeger} observed that a set $\mb \subset \binom A4$ constitutes a $3$-$(8,4,1)$ design on $A \setminus \{a\}$ if and only if $\mb \cup \{ A \setminus \{a\}\}$ are the points of a generator in $\mq^+(7,2)$.
This in particular yields a one-to-one correspondence between spreads of $\mq^+(7,2)$ and $\OS(4,8)$'s.

The group $\Sym_9$ naturally acts on $V$ by coordinate permutation.
Coordinate permutation does not alter the quadratic form $f$, so we can see $\Sym_9$ as a subgroup of $\PGO^+(8,2)$.
Its action on the spreads of $\mq^+(7,2)$ coincides with the action of $\Sym_9$ on the $\OS(4,8)$'s induced by the natural action of $\Sym_9$ on the 9-element ground set.
Since every system of generators has the same number of spreads, and the action of $\Sym_9$ has orbits of sizes $240$ and $1680$ on the spreads of $\mq^+(7,2)$, it cannot stabilise the systems of generators.
Thus, its subgroup stabilising the systems of generators must be of index 2, hence be $\Alt_9$.
We know the orbits of $\Alt_9$ on the $\OS(4,8)$'s.
Hence, in each system of generators, the orbits of $\Alt_9$ on the spreads have sizes $120$, corresponding to $\OS(4,8)$'s of type (1), and $840$, corresponding to $\OS(4,8)$'s of type (2).
Moreover, as we argued, the subgroup of $\Alt_9$ that setwise stabilises a pair of coordinate positions has the same orbits on the spreads of $\mq^+(7,2)$.

\subsection{The association scheme}
 \label{Sec:Assoc:Spreads}

We will now present an association scheme defined on $120$ $\OS(4,8)$'s of type (1) that form an orbit under the action of $\Alt_9$.
However, we will rephrase everything to the language of spreads of $\mq^+(7,2)$.
We start with the following simple observation.

\begin{lm}
 Let $\ms$ be a spread of $\mq^+(7,2)$, and $\Sigma$ a generator from the same system as the generators of $\ms$.
 Then either $\Sigma \in \ms$, or $\Sigma$ intersects five generators of $\ms$ in a line, and is disjoint from the other four.
\end{lm}

\begin{proof}
 Each generator $\Sigma_i$ of $\ms$ lives in the same system of generators as $\Sigma$, hence intersects it in a subspace of dimension $-1$, $1$, or $3$.
 Moreover, every point of $\Sigma$ is on a unique generator of $\ms$.
 Hence, if $\Sigma \notin \ms$, the non-empty intersections of $\Sigma$ with the elements of $\ms$ partition the points of $\Sigma$ into lines.
 Since $\Sigma$ has 15 points, and a line has three points, such a partition contains five lines.
\end{proof}

Let $V$ and $f$ be as in (\ref{Eq:CP}).
Let $Q_i$ denote the point whose coordinate vector has a single 0 in position $i$.
Consider $\Alt_9$ acting on $V$ by coordinate permutation, and let $G^*$ denote the subgroup of $\Alt_9$ that setwise stabilises coordinates 1 and 2.
Then we saw that in each system of generators of $\mq^+(7,2)$, $G^*$ has two orbits of sizes $120$ and $840$ on the spreads.
Let $\mx$ be such an orbit of size $120$.

\begin{prop}[GAP]
 \label{Prop:Gen Trans}
 The action of $G^*$ on $\mx$ is generously transitive and has rank $7$.
\end{prop}

Next, we introduce two quantities that are invariant under the action of $G^*$.
Given a spread $\ms \in \mx$, let $\ms_i$ denote the generator in $\ms$ containing the point $Q_i$.
Given two spreads, $\ms$ and $\mt$, the number
\[
 \Inv_1(\ms,\mt) = \dim(\ms_1 \cap \mt_2) + \dim(\ms_2 \cap \mt_1)
\]
is easily checked to be $G^*$-invariant.
We define $\Inv_2(\ms,\mt)$ to be the number of generators in $\{\mt_3, \dots, \mt_9\}$ that intersect both $\ms_1$ and $\ms_2$ in a line.
This is also easily checked to be $G^*$-invariant.
Note that by $G^*$ acting generously transitively, we know that $\Inv_i(\ms,\mt) = \Inv_i(\mt,\ms)$ for $i = 1,2$.

We define relations $R_0, \dots, R_6 \subset \mx^2$, where $R_0$ is simply the identity relation, and relations $R_1, \dots, R_6$ are defined using the Table \ref{tab:Assoc1}.

\begin{table}[h!]
 \centering
 \begin{tabular}{|l||c | c | c | c|}
  \hline
  & $\Inv_2(\ms,\mt) = 0$ & $\Inv_2(\ms,\mt) = 1$ & $\Inv_2(\ms,\mt) = 2$ & $\Inv_2(\ms,\mt) = 3$  \\ \hline \hline
  $\Inv_1(\ms,\mt) = -2$ & $R_2$ & & $R_3$ & \\ \hline
  $\Inv_1(\ms,\mt) = 0$ & & $R_5$ & & $R_6$ \\ \hline
  $\Inv_1(\ms,\mt) = 2$ & $R_1$ & & $R_4$ & \\ \hline
 \end{tabular}
 \caption{Relations of the association scheme on $120$ spreads of $\mq^+(7,2)$.}
 \label{tab:Assoc1}
\end{table}

\begin{thm}[GAP]
 \label{Thm:Assoc:Spreads}
 The 6-class symmetric association scheme on $\mx$, formed by the orbitals of the generously transitive action of $G^*$, is given by the relations $R_0, \dots, R_6$ as described in Table \ref{tab:Assoc1}.
 The matrix of eigenvalues is given by
 \[
  \BP = \begin{pmatrix}
1 & 7 & 14 & 21 & 21 & 42 & 14\\
1 & 2 & 4 & -9 & 6 & -3 & -1\\
1 & -1 & -2 & 3 & 3 & -6 & 2\\
1 & -4 & -2 & -3 & 0 & 9 & -1\\
1 & 3 & -2 & 3 & -1 & 2 & -6\\
1 & -2 & 8 & 3 & -6 & -3 & -1\\
1 & 3 & -2 & -3 & -7 & 2 & 6\\
\end{pmatrix}
 \]
 There are two possible fusions of the relations that yield an $\SRG(120,56,28,24)$:
\begin{align*}
 \BP (0,0,0,0,0,1,1)^\top
 = (56,-4,8,-4,-4,-4,8)^\top, &&
 \BP (0,0,0,1,1,0,1)^\top
 = (56,-4,-4,-4,8,-4,-4)^\top.
\end{align*}
 The graph given by $R_5 \cup R_6$ is isomorphic to $\nono$.
 The graph $\Gamma^*$ given by $R_3 \cup R_4 \cup R_6$ has $G^*$ as its full automorphism group.
\end{thm}

\begin{rmk}[GAP]
 \label{Rmk:extra invariant}
Another natural $G^*$-invariant is of course checking how many generators $\ms$ and $\mt$ have in common, and we can even split this into 2 $G^*$-invariants: $|\{\ms_1, \ms_2\} \cap \{\mt_1, \mt_2\}|$ and $|\{\ms_3, \dots, \ms_9\} \cap \{\mt_3, \dots, \mt_9\}|$.
If we replace $\Inv_1$ or $\Inv_2$ by this invariant, we can no longer distinguish all the relations of the association scheme.
However, this invariant gives useful information.
We describe it in Table \ref{tab:extra invariant}.
 Note that two spreads $\ms$ and $\mt$ in $\mx$ share at most one generator.
 Moreover, we can define $\nono$ as the graph on $\mx$ where two spreads are adjacent if they share no generator.

\begin{table}[h!]
 \centering
 \begin{tabular}{|c||c|c|}
 \hline
 & $|\{\ms_3, \dots, \ms_9\} \cap \{\mt_3, \dots, \mt_9\}| = 0$ & $|\{\ms_3, \dots, \ms_9\} \cap \{\mt_3, \dots, \mt_9\}| = 1$ \\ \hline \hline
 $|\{\ms_1, \ms_2\} \cap \{\mt_1, \mt_2\}| = 0$ & $R_5$, $R_6$ & $R_1$, $R_3$, $R_4$ \\ \hline
 $|\{\ms_1, \ms_2\} \cap \{\mt_1, \mt_2\}| = 1$ & $R_2$ & \\ \hline
 \end{tabular}
 \caption{An extra invariant on the association scheme.}
 \label{tab:extra invariant}
\end{table}
\end{rmk}

\section{Ovoids and anisotropic points \texorpdfstring{of $\mq^+(7,2)$}{}}

If we represent $\mq^+(7,2)$ using the vector space and quadratic form from (\ref{Eq:CP}), then the points $Q_1, \dots, Q_9$ whose coordinate vector has weight 8 form an ovoid $\mo$ of $\mq^+(7,2)$.
Clearly, the group $\Sym_9 \leq \PGO^+(8,2) \leq \PGL(8,2)$ of coordinate permutations stabilises $\mo$ setwise.
The ovoid is a frame and since $\PGL(8,2)$ acts sharply transitively on ordered frames, we see that $\Sym_9$ is the full setwise stabiliser of $\mo$ in $\PGO^+(8,2)$.
Moreover, $\Alt_9$ is the setwise stabiliser of $\mo$ that also stabilises the systems of generators.

All ovoids of $\mq^+(7,2)$ are isomorphic, see e.g.\ \cite{Cameron:Praeger}.
Thus, we can give a coordinate-free definition of the group $G^*$ from the previous section, by letting it be the setwise stabiliser of any ovoid of $\mq^+(7,2)$ that stabilises the systems of generators.
Using the principle of triality, we can translate the results of the previous section to the following proposition.

\begin{prop}
 \label{Prop:by triality}
 Let $\ms$ be a spread of $\mq^+(7,2)$, and let $\Sigma_1, \Sigma_2 \in \ms$.
 Let $G^*$ be the setwise stabiliser of $\{\Sigma_1, \Sigma_2\}$ and $\ms \setminus \{\Sigma_1, \Sigma_2\}$ in $\PGO^+(8,2)$.
 Then $G^*$ has two orbits of sizes $120$ and $840$ on the ovoids of $\mq^+(7,2)$.
 Let $\my$ denote the orbit of size $120$.
 The action of $G^*$ on $\my$ is generously transitive of rank $7$.
 Moreover, distinct ovoids in $\my$ share at most one point.
\end{prop}

This means that any ovoid $\mo \in \my$ is uniquely determined by the points $\mo \cap \Sigma_1$ and $\mo \cap \Sigma_2$.
On the other hand, there are 15 choices for a point $P_1 \in \Sigma_1$, and 8 choices for a point $P_2 \in \Sigma_2 \setminus P_1^\perp$, which yields a total of $15 \cdot 8 = 120 = |\my|$ pairs of non-collinear points.
This gives a one-to-one correspondence.
On the other hand, since the vector space underlying $\pg(7,2)$ is a direct sum of the subspaces corresponding to $\Sigma_1$ and $\Sigma_2$, given a point $P = \vspan x \notin \Sigma_1 \cup \Sigma_2$, there is a unique way to write $x$ as a linear combination of a vector in $\Sigma_1$ and $\Sigma_2$ seen as vector subspaces.
This means that there is a unique line $\ell_P$ through $P$ that intersects both $\Sigma_1$ and $\Sigma_2$.
Note that given two points $P_1 \in \Sigma_1$ and $P_2 \in \Sigma_2$, the line $P_1 P_2$ contains one extra point $P_3$, and this point belongs to $\mq^+(7,2)$ if and only if $P_1 \perp P_2$.
Therefore, there is a one-to-one correspondence between anisotropic points and pairs of non-collinear points in $\Sigma_1, \Sigma_2$.

To summarise, for each anisotropic point $P$, there is a unique line $\ell_P$ through $P$ that intersects $\Sigma_1$ and $\Sigma_2$, and a unique ovoid $\mo_P \in \my$ that contains $\Sigma_1 \cap \ell_P$ and $\Sigma_2 \cap \ell_P$.
This gives a one-to-one correspondence between anisotropic points and the ovoids of $\my$, and this correspondence is $G^*$-invariant.

In particular, this allows us to translate the association scheme from Section \ref{Sec:Assoc:Spreads} to an association scheme on the anisotropic points.
The relations of this scheme are best described by not only considering the anisotropic points, but also their corresponding ovoids.
We will denote the ovoid of $\my$ corresponding to the anisotropic point $P$ by $\mo_P$.
Given anisotropic points $P$ and $Q$, a useful $G^*$-invariant property is the number of points in $\mo_P$ orthogonal to $Q$.
There are only two options.

\begin{lm}
 Let $\mo$ be an ovoid of $\mq^+(7,2)$, and $P$ an anisotropic point.
 Then one of the following holds:
 \begin{enumerate}
  \item Exactly one line through $P$ intersects $\mo$ in 2 points, and $|P^\perp \cap \mo| = 7$.
  \item No line through $P$ intersects $\mo$ in 2 points, and $|P^\perp \cap \mo| = 3$.
 \end{enumerate}
\end{lm}

\begin{proof}
 We can use the representation of (\ref{Eq:CP}) for $V$ and $f$.
 Let $e_1, \dots, e_9$ denote the standard basis vectors of $\FF_2^9$.
 We can apply a coordinate transformation such that $\mo$ consists of the weight 8 points.
 Let $P$ be an anisotropic point with coordinate vector $v$.
 If $v$ has weight $2$, then $v = e_i + e_j$.
 Given a vector $w = \one + e_k$ of weight $8$, $v+w$ has weight $8$ if $k \in \{i,j\}$, and weight $6$ otherwise.
 Hence, $P$ is orthogonal the points $\vspan{\one + e_k}$ with $k \notin \{i,j\}$, and $P$, $\vspan{\one + e_i}$, $\vspan{\one + e_j}$ lie on the same line.
 If $v$ has weight $6$, then $v = \one + e_i + e_j + e_k$.
 Given a vector $w = \one + e_h$ of weight 8, $v+w$ has weight 2 if $h \in \{i,j,k\}$, and weight 4 otherwise.
 Thus, no two points of weight 8 lie on the same line through $P$, and there are exactly 3 points of $\mo$ which lie on a tangent line to $\mq^+(7,2)$ through $P$, hence are orthogonal to $P$.
\end{proof}

We describe some other $G^*$-invariants that help us distinguish the $G^*$-orbitals.
Consider the intersection of $\mo_P$ and $\mo_Q$.
We know from Remark \ref{Rmk:extra invariant} that they can share at most one point.
We can distinguish whether or not this point belongs to $\Sigma_1 \cup \Sigma_2$ or not.
In addition, the graph on $\my$ with the adjacency relation of being disjoint yields the graph $\nono$.
Unsurprisingly, we checked with GAP that $\mo_P \cap \mo_Q = \varnothing$ if and only if $P \not \perp Q$.
Note also that if $\ell_P$ is the unique line through $P$ intersecting $\Sigma_1$ and $\Sigma_2$, then $\ell_P \cap \Sigma_1 = \vspan{P, \Sigma_2} \cap \Sigma_1$ and $\ell_P \cap \Sigma_2 = \vspan{P, \Sigma_1} \cap \Sigma_2$.
Thus, if $\mo_P$ and $\mo_Q$ intersect in a point $S \in \Sigma_i$ with $i \in \{1,2\}$, then $\vspan{P, \Sigma_{3-i}} = \vspan{Q, \Sigma_{3-i}}$, which means that $PQ$ intersects $\Sigma_{3-i}$.
This brings us to the second $G^*$-invariant to consider: whether $PQ$ intersects $\mo_P$ (in which case $PQ$ also intersects $\mo_Q$ by $G^*$ acting generously transitively, and hence $\mo_P$ and $\mo_Q$ intersect), whether it intersects $\Sigma_1 \cup \Sigma_2$ (which is equivalent to $\mo_P$ and $\mo_Q$ sharing a point of $\Sigma_1 \cup \Sigma_2$), or whether it intersects neither.
These invariants suffice to translate the association scheme of Section \ref{Sec:Assoc:Spreads} to an association scheme on the anisotropic points.

\begin{thm}[GAP]
 \label{Thm:Assoc:Ovoids}
 Let $\ms$, $\Sigma_1$, $\Sigma_2$, $G^*$, and $\my$ be as in Proposition \ref{Prop:by triality}.
 For each anisotropic point $P$, let $\mo_P$ be the corresponding ovoid of $\my$.
 Let $\mz$ be the set of anisotropic points of $\pg(7,2)$, let $R_0$ be the identity relation on $\mz$, and define the relations $R_1, \dots, R_6$ using the Table \ref{tab:Assoc2}.
 Then $\{R_0, \dots, R_6\}$ is the 6-class symmetric association scheme arising from the action of $G^*$ on $\mz$, or equivalently on $\my$.
 This association scheme is isomorphic to the association scheme from Theorem \ref{Thm:Assoc:Spreads}, respecting the order of the relations.
 The fusion $R_5 \cup R_6$ yields the graph $\nono$.
 The fusion $R_3 \cup R_4 \cup R_6$ yields another $\SRG(120,56,28,24)$ $\Gamma^*$ with $G^*$ as full automorphism group.

 \begin{table}[h!]
 \centering
 \begin{tabular}{|c||c|c|} \hline
   & $|\mo_P \cap Q^\perp| = 3$ & $|\mo_P \cap Q^\perp| = 7$ \\ \hline \hline
   $\mo_P \cap \mo_Q$ is a point in $\Sigma_1 \cup \Sigma_2$ & $R_2$ & \\
   $P \perp Q$ & & \\ \hline
   $\mo_P \cap \mo_Q$ is a point outside $\Sigma_1 \cup \Sigma_2$ & $PQ$ intersects $\mo_P $: $R_1$ & $R_4$ \\
   $P \perp Q$ & $PQ$ does not intersect $\mo_P$: $R_3$ & \\ \hline
   $\mo_P \cap \mo_Q = \varnothing$ & $R_5$ & $R_6$\\
   $P \not \perp Q$ & & \\ \hline
 \end{tabular}
 \caption{The relations of the 6-class association scheme on the anisotropic points of $\mq^+(7,2)$.}
 \label{tab:Assoc2}
\end{table}

\end{thm}

\section{Comparison with known constructions of \texorpdfstring{$\SRG(120,56,28,24)$}{}}

We listed several constructions of $\SRG(120,56,28,24)$'s in the introduction.
In this section, we compare $\Gamma^*$ with several of these constructions, and prove that the graphs are not isomorphic.
We can easily exclude the other known vertex-transitive $\SRG(120,56,28,24)$'s by the rank of the action of the automorphism group.
We will prove that $\Gamma^*$ cannot be obtained from the constructions of Wallis \cite{Wallis}, Goethals and Seidel \cite{Goethals:Seidel}, or as the non-collinearity graph of a partial geometry.
We do this by inspecting the maximum cocliques of $\Gamma^*$ using GAP.

An $\SRG(120,56,28,24)$ has eigenvalues $56, 8, -4$.
By the Hoffman coclique bound, see e.g.\ \cite[Proposition 1.1.7]{Brouwer:VanMaldeghem}, a coclique $C$ in an $\SRG(120,56,28,24)$ has size at most 
$120 /( \frac{56}4 +1) = 8$, and if equality holds, every vertex outside $C$ has $4$ neighbours inside $C$.

\begin{prop}[GAP]
 \label{Prop:Cocliques}
 Consider the $\SRG(120,56,28,24)$ $\Gamma^*$ as defined in Theorem \ref{Thm:Assoc:Ovoids}.
 Let $G^*$ be its automorphism group.
 \begin{enumerate}
  \item $G^*$ has $3$ orbits $\mc_1, \mc_2, \mc_3$ on the cocliques of $\Gamma^*$ of size $8$.
  Their sizes are $|\mc_1| = 30$, $|\mc_2| = 120$, and $|\mc_3| = 210$.
  \item \label{item:nb} If $C$ is a coclique in $\mc_1$, $\mc_2$, or $\mc_3$, the neighbourhoods of the vertices outside $C$ intersect $C$ in respectively $14$, $49$, or $26$ different ways.
  \item \label{item:C1} The cocliques in $\mc_1$ are of the following form: take a point $Q \in \Sigma_i$ with $i \in \{1,2\}$, and let $C$ consist of the 8 anisotropic points $P$ such that $PQ$ intersects $\Sigma_{3-i}$.
  Note that any two points in $C$ are in relation $R_2$.
  We denote this coclique by $C_{1,Q}$.

  If $Q$ and $Q'$ are distinct points of $\Sigma_i$, then all neighbourhoods of a vertex in $C_{1,Q'}$ intersect $C_{1,Q}$ in a different $4$-set.
  \item The cocliques in $\mc_2$ are of the following form:
  Take an anisotropic point $P$, and let $C$ consist of $P$ and the $7$ points in relation $R_1$ with respect to $P$.
  Any two points in $C \setminus \{P\}$ are in relation $R_5$.

  \item \label{item:pg} The largest collection of maximum cocliques in $\Gamma^*$ that pairwise share at most one vertex has size $66$.
 \end{enumerate}
\end{prop}

First consider the construction of Wallis \cite[Theorem 1]{Wallis}, \cite{WallisCorr} and the complement of the construction of Goethals and Seidel \cite[Theorem 2.4]{Goethals:Seidel} (see also e.g.\ \cite[\S 8.C]{Brouwer:VanLint}).
In both cases, one can construct an $\SRG(120,56,28,24)$ by starting from a $2$-$(15,3,1)$ design $\mb$ on ground set $A$, and taking as vertex set some set that naturally has the structure of $A \times B$ for some $8$-element set $B$.
For any $a \in A$, the induced subgraph on $\{a\} \times B$ is a coclique of size $8$, and for any block $\{a_1,a_2,a_3\} \in \mb$ the induced subgraph on $\{a_1,a_2,a_3\} \times B$ are two disjoint copies of the complete tripartite graph $K_{4,4,4}$.
This means that $\Pi = \sett{\{a\} \times B}{a \in A}$ is a partition of the vertices into 15 pairwise disjoint cocliques.
For every 4-element subset $D$ of a coclique $C \in \Pi$, the number of vertices outside $C$ whose neighbourhood intersects $C$ in $D$ is a multiple of $8$.
In particular, the neighbourhoods of the vertices outside $C$ intersect $C$ in at most $\frac{120-8}8 = 14$ different ways.

If such a partition $\Pi$ exists in $\Gamma^*$, then by Proposition \ref{Prop:Cocliques} (\ref{item:nb}), all cocliques in $\Pi$ belong to $\mc_1$.
It is easy to see that the only way to partition the vertices of $\Gamma^*$ into cocliques of $\mc_1$ is by taking the $C_{1,Q}$ where all $Q$ belong to the same generator $\Sigma_i$, $i \in \{1,2\}$.
But by Proposition \ref{Prop:Cocliques} (\ref{item:C1}), the induced subgraph on $C_{1,Q} \cup C_{1,Q'}$ is not isomorphic to 2 copies of $K_{4,4}$.
Hence, such a partition $\Pi$ does not exist in $\Gamma^*$.

Secondly, the complement graph $\Gamma$ of the collinearity graph of a partial geometry with parameters $(7,8,4)$ is also an $\SRG(120,56,28,24)$, see e.g.\ \cite[\S 8.6]{Brouwer:VanMaldeghem}.
Such a partial geometry has $135$ lines, which give $135$ cocliques of size $8$ in $\Gamma$ that pairwise intersect in at most one vertex.
We know from Proposition \ref{Prop:Cocliques} (\ref{item:pg}) that $\Gamma^*$ does not have such a collection of $135$ cocliques.

We conclude that $\Gamma^*$ cannot be constructed using the constructions of Wallis \cite{Wallis} or Goethals--Seidel \cite{Goethals:Seidel}, nor does it arise as the non-collinearity graph of a partial geometry.

\section*{Acknowledgements}

The first author was supported by grant 12A3Y25N of Research Foundation Flanders (FWO).  The second author is supported by an NSERC Discovery Grant.  The fourth author is supported by
the SFB-TRR 195 ``Symbolic Tools in Mathematics and their Application'' of the German Research Foundation (DFG).  The research in this paper was also partially supported by OZR3637 startup-bonus ``Discrete Structures, and their applications in Data Science''.
The second and fourth authors would like to thank the Vrije Universiteit Brussel for its hospitality. 

\newcommand{\etalchar}[1]{$^{#1}$}

\end{document}